\documentclass[11pt,a4paper]{amsart}
\usepackage{enumerate}
\usepackage{hyperref}
\usepackage{fullpage}
\usepackage{xcolor}

\title{Existential fragments of theories of\\henselian valued fields}
\author{Sylvy Anscombe}
\author{Arno Fehm}
\address{Universit\'{e} Paris Cit\'{e}, Sorbonne Universit\'{e}, CNRS, IMJ-PRG, F-75013 Paris, France}
\email{sylvy.anscombe@imj-prg.fr}
\address{Institut f\"{u}r Algebra, Technische Universit\"{a}t Dresden, 01062 Dresden, Germany}
\email{arno.fehm@tu-dresden.de}

\theoremstyle{plain}
\newtheorem{theorem}{Theorem}[section]
\newtheorem{corollary}[theorem]{Corollary}

\newtheorem{proposition}[theorem]{Proposition}
\newtheorem{lemma}[theorem]{Lemma}

\theoremstyle{remark}

\newtheorem{remark}[theorem]{Remark}

\newcommand{\Rfour}{$(\mathrm{R}4)$}
\newcommand{\RfourN}{$(\mathrm{R}4_n)$}
\newcommand{\RfourNplusone}{$(\mathrm{R}4_{n+1})$}
\newcommand{\Rfourfour}{$(\mathrm{R}4_4)$}

\newcommand{\Hen}{H^{e\prime}}
\newcommand{\Hpi}{H^{\varpi}}
\newcommand{\Hepi}{H^{e,\varpi}}

\newcommand\Lang{\mathfrak{L}}
\newcommand\Lring{\Lang_{\mathrm{ring}}}
\newcommand\Lfield{\Lang_{\mathrm{field}}}
\newcommand\Lval{\Lang_{\mathrm{val}}}
\newcommand\Th{\mathrm{Th}}

\newcommand\FF{\mathbb{F}}
\newcommand\NN{\mathbb{N}}
\newcommand\ZZ{\mathbb{Z}}

\newcommand\Gps[2]{{#1(\!(#2)\!)}}
\newcommand\iGps[2]{{#1[\![#2]\!]}}
\newcommand\Hs[2]{\Gps{#1}{t^{#2}}}
\newcommand\iHs[2]{\iGps{#1}{t^{#2}}}
\newcommand\ps[1]{\Hs{#1}{}}
\newcommand\ips[1]{\iHs{#1}{}}

\numberwithin{equation}{section}

\begin{document}

\begin{abstract}
We study fragments of the existential theory of henselian valued fields with parameters.
This includes the $\exists_n$-fragment in the equicharacteristic or unramified mixed characteristic case,
the $\exists_n\exists_1$-fragment in the equicharacteristic case,
and the $\exists^n$-fragment in the residue characteristic zero case.
For example, we obtain an unconditional axiomatization (and thereby decidability) of the $\exists_3$-theory of $\ps{\FF_{q}}$ in the language of valued fields with a parameter for $t$.
\end{abstract}

\maketitle

\section{Introduction}
\label{sec:intro}

\noindent
While the theories of the valued field of $p$-adic numbers
and of henselian valued fields of residue characteristic zero such as the Laurent series field $\mathbb{C}(\!(t)\!)$
are well-understood by the work of Ax--Kochen and Ershov,
the model theory of henselian valued fields of positive characteristic is much more mysterious.
In particular, there is no known axiomatization of
the theories of local fields of positive characteristic -- the Laurent series fields $\ps{\FF_{q}}$
over the finite field with $q$ elements.
The best results in this direction are for {\em existential} theories (we recall notation and definitions in Section \ref{sec:notation}):

\begin{theorem}[\cite{DS,AF16,ADF23}]
Let $\Hen$ be the $\mathfrak{L}_{\rm val}$-theory of equicharacteristic nontrivially henselian valued fields,
and let $\Hepi$ be the $\mathfrak{L}_{\rm val}(\varpi)$-theory of equicharacteristic henselian valued fields with distinguished uniformizer $\varpi$.
\begin{enumerate}
    \item If $(K,v),(L,w)\models\Hen$ and ${\rm Th}_\exists(Kv)\subseteq{\rm Th}_\exists(Lw)$, then ${\rm Th}_\exists(K,v)\subseteq{\rm Th}_\exists(L,w)$.
\item Suppose that
\begin{enumerate}
\item[\Rfour]
 If a valuation $u$ on a field $E$ is trivial on a large subfield $F$ of $E$ with residue field $Eu=F$, then $F\prec_{\exists}E$, i.e., $E$ and $F$ satisfy the same $\exists$-$\mathfrak{L}_{\rm ring}(F)$-sentences. 
\end{enumerate} 
If $(K,v,t),(L,w,s)\models\Hepi$
and ${\rm Th}_\exists(Kv)\subseteq{\rm Th}_\exists(Lw)$, then ${\rm Th}_\exists(K,v,t)\subseteq{\rm Th}_\exists(L,w,s)$.    
\end{enumerate}
\end{theorem}

Here, {\em large} is meant in the sense of Pop (see \cite{Pop,BSFsurvey,Popsurvey} for background on large fields),
and the assumption \Rfour\ is a consequence of resolution of singularities first explored (without the name) in \cite{Kuhlmann} and \cite{Feh11}, see \cite[Section 2]{ADF23}.
`Monotonicity' statements like these are key to obtaining axiomatizations, decidability and further model theoretic consequences across different residue fields; see for example \cite{ADF23} or \cite{AFfragments}, where both (1) and (2) prove instances of the crucial axiom (mon) (\cite[Definition 2.16]{AFfragments}). For example, one can deduce:

\begin{corollary}
	${\rm Th}_\exists(\ps{\FF_{q}},v_t)$ is decidable,
and if \Rfour\ holds, then also
${\rm Th}_\exists(\ps{\FF_{q}},v_t,t)$
is decidable.
\end{corollary}

Results like part (1) of the theorem hold also when restricted to certain fragments of the existential theory, like the $\exists_n$-fragment, consisting of positive boolean combinations of sentences of the form
$\exists x_1,\dots, x_n \psi$
with quantifier-free $\psi$:

\begin{theorem}[{\cite[Lemma 3.23]{AFfragments}}]
Let $(K,v),(L,w)\models\Hen$ and $n\in\NN$.
If ${\rm Th}_{\exists_n}(Kv)\subseteq{\rm Th}_{\exists_n}(Lw)$,
then ${\rm Th}_{\exists_n}(K,v)\subseteq{\rm Th}_{\exists_n}(L,w)$.
\end{theorem}

The main aim of this note is to prove statements of this form for $\Hepi$.
It turns out that not the full \Rfour\ is needed (see Theorem~\ref{thm:En}): 

\begin{theorem}\label{thm:intro}
Let $(K,v,t),(L,w,s)\models\Hepi$ and $n\in\NN$. Suppose that
\begin{enumerate}
 \item[\RfourN] If a valuation $u$ on a field $E$ is trivial on a large subfield $F$ of $E$ with residue field $Eu=F$, then $F\prec_{\exists_n}E$, i.e., $E$ and $F$ satisfy the same $\exists_n$-$\Lring(F)$-sentences.   
\end{enumerate}
If ${\rm Th}_{\exists_{n-1}}(Kv)\subseteq{\rm Th}_{\exists_{n-1}}(Lw)$,
then 
${\rm Th}_{\exists_{n-1}}(K,v,t)\subseteq{\rm Th}_{\exists_{n-1}}(L,w,s)$.
\end{theorem}

Since resolution of singularities is proven for threefolds by Cossart and Piltant \cite{CP}
and therefore \Rfourfour\ holds,
we obtain unconditional results for $\exists_3$ and $\exists_4$-theories,
in particular
explicit axiomatizations, from which we deduce
the following decidability results (Corollary \ref{cor:applications}, but see also Remark \ref{rem:DS}):

\begin{corollary}
	${\rm Th}_{\exists_3}(\ps{\FF_{q}},v_t,t)$ and
	${\rm Th}_{\exists_4}(\ps{\FF_{q}},t)$
are decidable.    
\end{corollary}

We also obtain similar results for mixed characteristic henselian valued fields, as well as for two further families of fragments, denoted $\exists_n\exists_1$ and $\exists^n$
(see Sections \ref{sec:EnE1} and \ref{sec:Euppern}).
Moreover, all our results are proven for constants from a subfield $C$ instead of just one constant.

\section{Notation}
\label{sec:notation}

\noindent
For definitions and notation regarding valued fields we generally follow \cite{EP}.
For a valued field $(K,v)$, we denote 
by $\mathcal{O}_v$ the valuation ring of $v$, 
by $vK$ the value group of $v$,
by $Kv$ the residue field of $v$,
by ${\rm res}_v\colon\mathcal{O}_v\rightarrow Kv$ the residue map,
and we write $\bar{a}={\rm res}_v(a)$.
We follow the convention to write $(F,v)$ instead of $(F,v|_F)$ for subfields $F$ of $K$.
A {\em uniformizer} of $(K,v)$ is an element $\pi$ of $K$
with $v(\pi)$ minimal positive,
and we say that $(K,v)$ is $\ZZ$-valued if $vK\cong\ZZ$.

As usual, $\Lring=\{+,-,\cdot,0,1\}$ is the language of rings.
We define the language of fields as $\Lfield=\mathfrak{L}_{\rm ring}\cup\{\cdot^{-1}\}$, where we always interpret $0^{-1}:=0$.
We work with the one-sorted language of valued fields
$\Lval=\Lring\cup\{\mathcal{O}\}$,
where in a valued field $(K,v)$, $\mathcal{O}$ is interpreted as the valuation ring $\mathcal{O}_v$ of $v$.
For a language $\Lang$ we write $\Lang(\varpi)$ for $\Lang$ expanded by a constant symbol $\varpi$,
and if $C$ is a field we denote by
$\Lang(C)$ 
the language $\Lang$
expanded by
constant symbols for the elements of $C$.
We write $\mathfrak{L}_{\rm val}$-structures as $(K,v)$, $\mathfrak{L}_{\rm val}(\varpi)$-structures as $(K,v,\pi)$,
and $\mathfrak{L}_{\rm val}(\varpi,C)$-structures as $(K,v,\pi,C)$,
or, if $\pi\in C$, often just as $(K,v,C)$.
We denote by 
\begin{itemize}
\item $\Hen$ the $\Lval$-theory of equicharacteristic henselian nontrivially valued fields, by
\item $\Hpi$ the $\Lval(\varpi)$-theory of henselian valued fields with distinguished uniformizer, and by
\item $\Hepi$ the $\Lval(\varpi)$-theory of equicharacteristic henselian valued fields with distinguished uniformizer.
\end{itemize}
We work with the fragments $\exists$ and $\exists_n$, as introduced in Section~\ref{sec:intro} or in more detail in \cite[Definition 2.3]{AFAE}, and use the notation introduced in \cite[Definition 2.1]{AFAE}: 
To be precise, for a language $\mathfrak{L}$, an {\em $\exists_n$-$\mathfrak{L}$-formula} is an $\mathfrak{L}$-formula of the form $\exists x_1,\dots,x_m\psi$ with $m\leq n$ and $\psi$ a quantifier-free $\mathfrak{L}$-formula, 
and an $\exists$-$\mathfrak{L}$-formula is an $\exists_n$-$\mathfrak{L}$-formula for some $n$,
but ${\rm Th}_{\exists}(K)$ respectively ${\rm Th}_{\exists_n}(K)$ denotes the set of {\sc positive boolean combinations} of $\exists$- respectively $\exists_n$-sentences that hold in the structure $K$, and then similarly for other fragments.

\section{Translating $\Lval$ to $\Lring$}

\noindent
It is well-known that every $\Lval$-formula $\varphi$ is equivalent modulo $\Hpi$ to a 
$\Lring(\varpi)$-formula $\varphi'$,
and that moreover when $\varphi$ is existential, also $\varphi'$ can be chosen to be existential \cite[Lemma 4.9]{ADF23}.
We need the slightly stronger statement that 
$\varphi'$ needs only one additional existential quantifier,
which we prove in this section.

\begin{lemma}\label{lem:field_to_ring}
Every quantifier-free $\Lfield$-formula $\varphi$ is equivalent modulo the $\Lfield$-theory of fields to a quantifier-free $\Lring$-formula.    
\end{lemma}

\begin{proof}
This can be seen by manipulating formulas, but many case distinctions are needed due to the fact that $\cdot^{-1}$ is not a homomorphism on the whole field (no matter how $0^{-1}$ is defined), and it also follows by general principles from the fact that every homomorphism from an integral domain $R$ to a field extends uniquely to the quotient field of $R$, cf.~\cite[Remark 3.1]{DDF}.
\end{proof}

\begin{lemma}\label{lem:def}
For every $n\in\NN$
there exists an $\exists_1$-$\Lring(\varpi)$-formula $\eta_n$
that defines 
$\mathcal{O}_v^n$
in every $(K,v,\pi)\models\Hpi$,
such that the map $n\mapsto\eta_n$ is computable.
\end{lemma}
\begin{proof}
The $\Lring(\varpi)$-formula
$$
 \exists z(z^2+z = \varpi x^2)
$$
defines $\mathcal{O}_v$ in every $(K,v,\pi)\models\Hpi$,
cf.~\cite[Lemma 4.9]{ADF23}.
So since $\bigwedge_{i=1}^nx_i\in\mathcal{O}_v$
is equivalent modulo $\Hpi$ to
$\sum_{i=1}^nx_i^n\varpi^i\in\mathcal{O}_v$, the $\exists_1$-$\Lring(\varpi)$-formula 
$$
 \eta_n := \exists z\Big(z^2+z=\varpi\big(\sum\nolimits_{i=1}^nx_i^n\varpi^i\big)^2\Big)
$$
is as claimed.
\end{proof}

\begin{proposition}\label{prop:val_to_ring}
Every $\exists_n$-$\Lval(\varpi)$-formula 
is equivalent modulo $\Hpi$ to 
an $\exists_{n+1}$-$\Lring(\varpi)$-formula.
\end{proposition}

\begin{proof}
It suffices to consider the case $n=0$.
So let $\varphi(\underline{x})$ be a quantifier-free $\Lval(\varpi)$-formula.
Without loss of generality, $\varphi$ is in disjunctive normal form, and is of the following form:
$$
\bigvee_{i=1}^r\bigwedge_{j=1}^{s_i} (f_{ij}(\underline{x})=0 \wedge 
g_{ij}(\underline{x})\neq 0 \wedge
h_{ij}(\underline{x})\in\mathcal{O}\wedge
k_{ij}(\underline{x})\notin\mathcal{O})
$$
with $f_{ij},g_{ij},h_{ij},k_{ij}\in\ZZ[\underline{x}]$.
Let $\eta_k=\exists z\rho_k$ be the formula from Lemma~\ref{lem:def}.
As 
$$ 
 \Hpi\models\forall y\big( y\notin\mathcal{O}\leftrightarrow ((y\varpi)^{-1}\in\mathcal{O}\wedge y\neq 0)\big),
$$ 
$\varphi$ is equivalent modulo $\Hpi$ to
the $\exists_1$-$\Lfield(\varpi)$-formula
$$
\exists z
\bigvee_{i=1}^r
\Big(
\rho_{2s_i}(h_{i1}(\underline{x}),\dots,h_{is_i}(\underline{x}),
(k_{i1}(\underline{x})\varpi)^{-1},
\dots,
(k_{is_i}(\underline{x})\varpi)^{-1},z)\wedge
\bigwedge_{j=1}^{s_i} (f_{ij}(\underline{x})=0 \wedge 
g_{ij}k_{ij}(\underline{x})\neq 0)\Big).
$$
By Lemma~\ref{lem:field_to_ring},
this is equivalent modulo $\Hpi$
to an $\exists_1$-$\Lring(\varpi)$-formula.
\end{proof}

\begin{remark}
Note that as $\Hpi$ has a recursive axiomatization, the $\exists_{n+1}$-formula from Proposition~\ref{prop:val_to_ring} can be found effectively, i.e.~there is a computable map
$\varphi\mapsto\varphi'$ from $\exists_n$-$\Lval(\varpi)$-formulas 
to
$\exists_{n+1}$-$\Lring(\varpi)$-formulas
such that $\Hpi\models\varphi\leftrightarrow\varphi'$.
\end{remark}

\section{Monotonicity for the $\exists_n$-fragment}

\noindent
We adapt the proof of \cite[Proposition 4.11]{ADF23} from $\exists$ to $\exists_n$, 
with special attention to which part of \Rfour\ is needed for the argument.
Theorem \ref{thm:intro} from the introduction will be a special case of Theorem~\ref{thm:En}(1) for $C=\mathbb{F}_p(t)$.

\begin{proposition}\label{prop:same_E4}
Let $(K,v)$ be an equicharacteristic henselian valued field with $Kv$ large,
and assume that $v$ is trivial on a subfield $C$ of $K$.
We identify $C$ with its image under ${\rm res}_v$
and assume that the extension $Kv/C$ is separable.
If \RfourN\ holds,
then ${\rm Th}_{\exists_n}(K,C)={\rm Th}_{\exists_n}(Kv,C)$.
\end{proposition}

\begin{proof}
The proof of \cite[Corollary 4.6]{ADF23}    
goes through with the fragment $\exists$ replaced by $\exists_n$, except that instead of \Rfour\
we apply \RfourN:
There is an elementary extension $(K,v)\prec(K^*,v^*)$ with a section $\zeta\colon K^*v^*\rightarrow K^*$ of ${\rm res}_{v^*}$ with $\zeta|_C={\rm id}_C$, and $Kv\prec K^*v^*$ are both large,
so \RfourN\ applied to the extension $K^*/\zeta(K^*v^*)$ gives that
${\rm Th}_{\exists_n}(K,C)={\rm Th}_{\exists_n}(\zeta(K^*v^*),C)={\rm Th}_{\exists_n}(Kv,C)$.
\end{proof}

\begin{lemma}\label{prop:rank_one}
Let $(C,u)$ be an equicharacteristic $\ZZ$-valued field with uniformizer $\pi$ such that $\mathcal{O}_u$ is excellent.
Let $(K,v)$ and $(L,w)$
be complete $\ZZ$-valued extensions of $(C,u)$ with uniformizer $\pi$ such that $Kv/Cu$ and $Lw/Cu$ are separable.
If ${\rm Th}_\exists(Kv,Cu)\subseteq{\rm Th}_\exists(Lw,Cu)$,
then ${\rm Th}_{\exists}(K,C)\subseteq{\rm Th}_{\exists}(L,C)$.
\end{lemma}

\begin{proof}
This is proven in the third paragraph of the proof of \cite[Proposition 4.11]{ADF23}.    
\end{proof}

\begin{lemma}\label{lem:reduce_to_rank_one}
Let $(C,u)$ be an equicharacteristic $\ZZ$-valued field with uniformizer $\pi$ such that $\mathcal{O}_u$ is excellent.
Let $(K,v)$ be a henselian extension of $(C,u)$ with uniformizer $\pi$ such that $Kv/Cu$ is separable.
If \RfourN\ holds,
then there exists an extension $(K',v',\pi)\models\Hpi$ of $(C,u,\pi)$ which is complete $\ZZ$-valued such that
${\rm Th}_{\exists_n}(K',C)={\rm Th}_{\exists_n}(K,C)$
and $Kv\prec K'v'$. 
\end{lemma}

\begin{proof}
This is proven 
exactly as the corresponding statement for \Rfour\
is proven in the second paragraph of the proof of \cite[Proposition 4.11]{ADF23}:
Let $(K,v)\prec (K^*,v^*)$ be an $\aleph_1$-saturated elementary extension, let $K'$ be the residue field of the finest proper coarsening $v^+$ of $v^*$,
note that $v^+|_C$ is trivial so that we can view $C$ as a subfield of $K'$,
and let $v'$ be the valuation induced by $v^*$ on $K'$.
Then $K'v'=K^*v^*\succ Kv$,
${\rm Th}_{\exists_n}(K',C)={\rm Th}_{\exists_n}(K^*,C)={\rm Th}_{\exists_n}(K,C)$
by Proposition~\ref{prop:same_E4},
and $(K',v')$ is complete by the saturation assumption \cite[Lemma 7.14]{Dries}.
\end{proof}

\begin{proposition}\label{prop:main}
Let $(C,u)$ be an equicharacteristic $\ZZ$-valued field with uniformizer $\pi$ such that $\mathcal{O}_u$ is excellent.
Let $(K,v)$ and $(L,w)$
be henselian extensions of $(C,u)$ with uniformizer $\pi$ such that $Kv/Cu$ and $Lw/Cu$ are separable.
If ${\rm Th}_\exists(Kv,Cu)\subseteq{\rm Th}_\exists(Lw,Cu)$ and \RfourNplusone\ holds,
then ${\rm Th}_{\exists_{n+1}}(K,C)\subseteq{\rm Th}_{\exists_{n+1}}(L,C)$
and
${\rm Th}_{\exists_n}(K,v,C)\subseteq{\rm Th}_{\exists_n}(L,w,C)$.
\end{proposition}

\begin{proof}
Let $\varphi\in{\rm Th}_{\exists_{n+1}}(K,C)$.
We have to show that $\varphi\in{\rm Th}_{\exists_{n+1}}(L,C)$.
By the assumption \RfourNplusone,
we can assume without loss of generality
that $(K,v)$ and $(L,w)$ are complete $\ZZ$-valued (Lemma~\ref{lem:reduce_to_rank_one}).
Lemma \ref{prop:rank_one} then shows that $\varphi\in {\rm Th}_{\exists_{n+1}}(L,C)$.
This shows that 
${\rm Th}_{\exists_{n+1}}(K,C)\subseteq{\rm Th}_{\exists_{n+1}}(L,C)$.
Finally, every $\psi\in{\rm Th}_{\exists_n}(K,v,C)$ is equivalent in both $(K,v,C)$ and $(L,w,C)$ to some $\varphi\in{\rm Th}_{\exists_{n+1}}(K,C)$ by Proposition \ref{prop:val_to_ring},
and therefore also ${\rm Th}_{\exists_n}(K,v,C)\subseteq{\rm Th}_{\exists_n}(L,w,C)$.
\end{proof}

We now adapt the preceding
proposition
to mixed characteristic unramified henselian valued fields.
This argument is almost quotable from the literature,
e.g.~from \cite[Corollary 5.7]{ADJ24},
but that statement does not allow parameters from a $\ZZ$-valued subfield.
We now briefly repeat the arguments from
\cite{ADJ24},
albeit with minor variations.

\begin{proposition}\label{prop:main_mixed_characteristic}
Let $(C,u)$ be a $\ZZ$-valued field of mixed characteristic with uniformizer $p$, i.e.~unramified.
Let $(K,v)$ and $(L,w)$
be henselian extensions of $(C,u)$, also with uniformizer $p$, such that $Kv/Cu$
is
	separable.
If ${\rm Th}_\exists(Kv,Cu)\subseteq{\rm Th}_\exists(Lw,Cu)$,
then
${\rm Th}_{\exists}(K,v,C)\subseteq{\rm Th}_{\exists}(L,w,C)$.
\end{proposition}
\begin{proof}
	Let $(K,v)\prec(K^{*},v^{*})$ be an $\aleph_{1}$-saturated elementary extension,
    let $K'=K^*v^+$ be the residue field of the finest proper coarsening $v^{+}$ of $v^{*}$,
	and note that $v^+|_C$ is trivial, so we can identify $C$ with a subfield of $K'$.
	First we claim that $\Th_{\exists}(K',v',C)=\Th_{\exists}(K^{*},v^{*},C)$,
	where $v'$ is the valuation induced on $K'$ by $v^{*}$.
	Indeed, since $v^{+}$ is henselian and of equal characteristic zero,
    ${\rm id}_C$ extends to a section $\zeta\colon K'\rightarrow K^{*}$ of ${\rm res}_{v^{+}}$,
	and this yields an embedding $(K',v',C)\rightarrow(K^{*},v^{*},C)$ of $\Lval(C)$-structures.
	By the saturation hypothesis, $(K',v',C)$ is complete \cite[Lemma 7.14]{Dries}.
    As $K'v'=K^*v^*$ and $v'K'=\mathbb{Z}v(p)\prec_\exists v^*K^*$ \cite[Corollary 4.1.4]{PD},
	the claim then follows from \cite[Theorem 5.14]{ADJ24},
	which shows that
	$(K',v')\prec_{\exists}(K^{*},v^{*})$.
	Applying the same argument to $(L,w)$,
	we may henceforth assume without loss of generality that both $(K,v)$ and $(L,w)$ are $\ZZ$-valued and complete,
	and even that $Lw$ is $|Kv|^{+}$-saturated or finite.
	  In the former case, by the assumption ${\rm Th}_\exists(Kv,Cu)\subseteq{\rm Th}_\exists(Lw,Cu)$, there is an $\Lring$-embedding $\phi\colon Kv\rightarrow Lw$ over $Cu$ \cite[Lemma 5.2.1]{CK};
    in the latter case, the assumption implies that also $Kv$ is finite, in particular algebraic over $Cu$, and we get  an $\Lring$-embedding $\phi\colon Kv\rightarrow Lw$ over $Cu$ also in this case (see e.g.~Lemma \ref{lem:E1} below).
	Since we have assumed that $Kv/Cu$ is separable, and that $(L,w)$ is complete,
	we may apply \cite[Lemma 4.11]{ADJ24}
	to obtain an $\Lval$-embedding $\Phi\colon(K,v)\rightarrow(L,w)$ over $(C,u)$,
	which finishes the proof.
\end{proof}

We prove two lemmas that we will use several times, and then state the main result of this section.

\begin{lemma}\label{lem:prim_elem}
Let $E/F$ be a finitely generated separable field extension.
Then $E$ is generated over $F$ by at most ${\rm trdeg}(E/F)+1$ many elements.
\end{lemma}

\begin{proof}
There exists a separating transcendence basis $t_1,\dots,t_d$ of $E/F$, 
and $E/F(t_1,\dots,t_d)$ is finite and separable, so the claim follows from the primitive element theorem.    
\end{proof}

\begin{lemma}\label{lem:E1}
Let $K,L$ be field extensions of a field $C$. 
The following are equivalent:
\begin{enumerate}
    \item ${\rm Th}_{\exists_1}(K,C)\subseteq{\rm Th}_{\exists_1}(L,C)$
    \item
\begin{enumerate}
 \item There exists a $C$-embedding of the relative algebraic closure $\overline{C}^K$ of $C$ in $K$ into $L$, and
 \item if $K$ is infinite, then so is $L$.
\end{enumerate}
\end{enumerate}
\end{lemma}

\begin{proof}
$(1)\Rightarrow(2a)$: See \cite[Corollary 20.6.4]{FJ}.

$(1)\Rightarrow(2b)$: If $K$ is infinite of characteristic $p$, then $p=0\wedge\exists x(x^{p^n}-x\neq0)\in{\rm Th}_{\exists_1}(K)$ for every $n$,
hence if ${\rm Th}_{\exists_1}(K,C)\subseteq{\rm Th}_{\exists_1}(L,C)$, then also $L$ is of characteristic $p$ and infinite.

$(2)\Rightarrow(1)$: Let $\varphi\in{\rm Th}_{\exists_1}(K,C)$.
Without loss of generality, $\varphi$ is of the form $\exists x\psi(x)$ with $\psi$ of the form $g(x)\neq0\wedge\bigwedge_{i=1}^n f_i(x)=0$ for $g,f_1,\dots,f_n\in C[x]$.
Let $a\in K$ such that $K\models\psi(a)$, and let $\alpha\colon\overline{C}^K\rightarrow L$ be the embedding given by (2a).
If $a$ is algebraic over $C$, then $L\models\psi(\alpha(a))$;
otherwise, $K$ is infinite, $g\neq 0$ and $f_i=0$ for all $i$,
and then also $L$ is infinite by (2b), hence there exists $b\in L$ with $g(b)\neq 0$, and therefore $L\models\psi(b)$.
\end{proof}

\begin{theorem}\label{thm:En}
Let $(C,u)$ be a $\mathbb{Z}$-valued field with uniformizer $\pi$,
and let $(K,v)$ and $(L,w)$ be henselian extensions of $(C,u)$ with uniformizer $\pi$ such that $Kv/Cu$ and $Lw/Cu$ are separable.
\begin{enumerate}
\item
Assume that $\mathcal{O}_u$ is equicharacteristic and excellent.
\begin{enumerate}
\item\label{case1}
If ${\rm Th}_{\exists_n}(Kv,Cu)\subseteq{\rm Th}_{\exists_n}(Lw,Cu)$
and \RfourN\ holds,
then ${\rm Th}_{\exists_n}(K,C)\subseteq{\rm Th}_{\exists_n}(L,C)$.
\item\label{case2}
If ${\rm Th}_{\exists_n}(Kv,Cu)\subseteq{\rm Th}_{\exists_n}(Lw,Cu)$ and \RfourNplusone\ holds,
then ${\rm Th}_{\exists_n}(K,v,C)\subseteq{\rm Th}_{\exists_n}(L,w,C)$.
\end{enumerate}
\item
Assume that $\mathcal{O}_u$ is of mixed characteristic and $\pi=p={\rm char}(Cu)$.
\begin{enumerate}
 \item\label{case3} If ${\rm Th}_{\exists_n}(Kv,Cu)\subseteq{\rm Th}_{\exists_n}(Lw,Cu)$, 
then ${\rm Th}_{\exists_n}(K,v,C)\subseteq{\rm Th}_{\exists_n}(L,w,C)$.
\end{enumerate}
\end{enumerate}
\end{theorem}

\begin{proof}
We prove \ref{case1}, \ref{case2} and \ref{case3} simultaneously
and follow the idea of \cite[Lemma 3.23]{AFfragments}.
It suffices to show that ${\rm Th}_{\exists_n}(M')\subseteq{\rm Th}_{\exists_n}(L,w,C)$ for every substructure $M'$ of $(K,C)$ respectively of $(K,v,C)$ generated by at most $n$ elements $a_1,\dots,a_n$, cf.~\cite[Lemma 3.22]{AFfragments}.
Let $R_0=C[a_1,\dots,a_n]$ be the subring they generate over $C$.
Let $E:={\rm Quot}(R_0)$ 
and note that 
$$
 {\rm Th}_{\exists_n}(M')\subseteq{\rm Th}_{\exists_n}(E,v,C)\subseteq{\rm Th}_{\exists_n}(E^h,v^h,C),
$$ 
where $(E^h,v^h)$ denotes the henselization of $(E,v)$ in $(K,v)$.
We will show that
\begin{equation}\label{eqn:1}
 {\rm Th}_\exists(Ev,Cu)\subseteq{\rm Th}_\exists(Lw,Cu),
\end{equation}
which will conclude the proof in all three cases:
In case \ref{case1},
by Proposition~\ref{prop:main} (using \RfourN), \eqref{eqn:1} will imply that
${\rm Th}_{\exists_n}(E^h,C)\subseteq{\rm Th}_{\exists_n}(L,C)$.
In case \ref{case2},
similarly by Proposition~\ref{prop:main} (now using \RfourNplusone), \eqref{eqn:1} will imply that
${\rm Th}_{\exists_n}(E^h,v^h,C)\subseteq{\rm Th}_{\exists_n}(L,w,C)$.
In case \ref{case3}, 
we apply instead Proposition~\ref{prop:main_mixed_characteristic}
to obtain even
$\Th_{\exists}(E^h,v^h,C)\subseteq\Th_{\exists}(L,w,C)$.

Now we prove \eqref{eqn:1}.
As $E=C(a_1,\dots,a_n)$, we have ${\rm trdeg}(E/C)\leq n$
and thus ${\rm trdeg}(Ev/Cu)\leq n$ by the Abyhankar inequality \cite[Theorem 3.4.3]{EP}. 
If ${\rm trdeg}(Ev/Cu)<n$, let $E':=Ev$.
Otherwise, ${\rm trdeg}(Ev/Cu)=n$ and $a_1,\dots,a_n$ are algebraically independent over $C$.
In particular, $E/C(a_1,\dots,a_{n-1})$ is a simple purely transcendental extension,
and then the ruled residue theorem \cite{Ohm}
implies that 
$Ev$ is a simple purely transcendental extension of 
a finite extension $E'$ of $C(a_1,\dots,a_{n-1})v$,
say $Ev=E'(t)$ with $t$ transcendental over $E'$.
If $E'$ is infinite, we have that
$E'\prec_\exists Ev$ \cite[Example 3.1.2]{Ershov}.
Therefore, except in the special case 
\begin{equation}\label{eqn:star}\tag{$\ast$}
n=1,\; Cu\mbox{ is finite, and }Ev=E'(t)\mbox{ with }E'/Cu\mbox{ finite }
\end{equation}
we have found $E'$ with ${\rm Th}_\exists(E',Cu)={\rm Th}_\exists(Ev,Cu)$ and ${\rm trdeg}(E'/Cu)<n$.

Let $R\subseteq E'$ be a subring that is a finitely generated extension of $Cu$.
Then $F:={\rm Quot}(R)$ is a subfield of $E'$ that is finitely generated over $Cu$.
Moreover, ${\rm trdeg}(F/Cu)\leq {\rm trdeg}(E'/Cu)<n$.
Since $Kv/Cu$ is separable, so is $F/Cu$, so since it is also finitely generated,
it is generated by at most $n$ elements (Lemma \ref{lem:prim_elem}), say $b_1,\dots,b_n$.
Then $S:=Cu[b_1,\dots,b_n]$ 
is a substructure of $(Kv,Cu)$ generated by $n$ elements,
so the assumption ${\rm Th}_{\exists_n}(Kv,Cu)\subseteq{\rm Th}_{\exists_n}(Lw,Cu)$ implies that
${\rm Th}_\exists(S,Cu)\subseteq{\rm Th}_\exists(Lw,Cu)$
\cite[Lemma 3.22(b)]{AFfragments}.
Since $F={\rm Quot}(S)$, this implies that
${\rm Th}_\exists(F,Cu)\subseteq{\rm Th}_\exists(Lw,Cu)$.
Since $R$ was an arbitrary finitely generated substructure of $(E',Cu)$,
this shows that ${\rm Th}_\exists(Ev,Cu)={\rm Th}_\exists(E',Cu)\subseteq{\rm Th}_\exists(Lw,Cu)$, hence \eqref{eqn:1} holds.

We finally argue that \eqref{eqn:1} also holds in the remaining special case \eqref{eqn:star}. 
The assumption
${\rm Th}_{\exists_{1}}(Kv,Cu)\subseteq{\rm Th}_{\exists_{1}}(Lw,Cu)$
implies that there exists an embedding of the relative algebraic closure of $Cu$ in $Kv$ into $Lw$ over $Cu$ (Lemma \ref{lem:E1}), in particular that there exists an embedding $\alpha\colon E'\rightarrow Lw$ over $Cu$.
Moreover, it also implies that since $Kv$ is infinite (as it contains $Ev=E'(t)$) and therefore contains for every $n$ an element that is not a zero of $X^{p^n}-X$, also $Lw$ is infinite. Thus $\alpha$ extends to an embedding $Ev\rightarrow Lw^*$ for an elementary extension $(Lw,Cu)\prec(Lw^*,Cu)$, and therefore
${\rm Th}_\exists(Ev,Cu)\subseteq{\rm Th}_\exists(Lw^*,Cu)={\rm Th}_\exists(Lw,Cu)$.
\end{proof}

\begin{remark}
Note that in Theorem \ref{thm:En}, the conclusion of (1a) for $n=n_0+1$ implies the conclusion of (1b) for $n=n_0$ (via Proposition~\ref{prop:val_to_ring}), but the assumption of (1a) for $n=n_0+1$ is stronger than the assumption of (1b) for $n=n_0$, so these seem to be two independent statements.
\end{remark}

\begin{remark}
At least part (1a) of Theorem~\ref{thm:En} is formulated under the weakest consequence of resolution of singularities under which it holds, in the sense that
if ${\rm Th}_{\exists_n}(Kv,Cu)\subseteq{\rm Th}_{\exists_n}(Lw,Cu)$
implies
${\rm Th}_{\exists_n}(K,C)\subseteq{\rm Th}_{\exists_n}(L,C)$ for every such $C,K,L$,
then \RfourN\ holds,
as can be seen by 
following and adapting the argument in
\cite[Remark 4.17]{ADF23}. 
\end{remark}

\begin{remark}\label{rem:char0}
If ${\rm char}(Cu)=0$, then the assumption \RfourN\ respectively \RfourNplusone\ in Theorem~\ref{thm:En}(1) can be removed, cf.~\cite[Remark 4.18, Remark 2.4]{ADF23} and see also \cite[Lemma 3.18(b)(iii)]{AFfragments}.   
\end{remark}

\section{Monotonicity for the $\exists_n\exists_1$-fragment}
\label{sec:EnE1}

\noindent
We now consider the fragment which in the notation introduced in \cite[Definition 2.3]{AFAE} is denoted $\exists_n\exists_1$: 
For a language $\mathfrak{L}$, 
an {\em $\exists_n\exists_1$-$\mathfrak{L}$-formula} is a formula of the form
$\exists x_1,\dots,x_m\psi$ where $m\leq n$ and $\psi$ is in the $\exists_1$-fragment,
 i.e.~a positive boolean combination of formulas of the form $\eta$ or $\exists y \eta$ with $\eta$ quantifier-free,
 and ${\rm Th}_{\exists_n\exists_1}(K)$ denotes the positive boolean combinations of $\exists_n\exists_1$-$\mathfrak{L}$-sentences that hold in $K$.
Note that every $\exists_{n+1}$-formula is an $\exists_n\exists_1$-formula but not vice versa.

In what follows we will make use of the machinery set up in \cite{DDF}, in particular the notion of {\em essential fiber dimension} ${\rm efd}_T(\varphi)$ of a formula with respect to a theory $T$ (\cite[Definition 4.1]{DDF}): 
For a field $F$, an $\Lring(F)$-theory $T$ of extensions of $F$, and an existential $\Lring(F)$-sentence $\varphi$, ${\rm efd}_T(\varphi)$ is the smallest $d$ such that if $\varphi$ holds in some $E\models T$, there exists a subextension $E_0$ of $E/F$ with ${\rm trdeg}(E_0/F)\leq d$ such that $\varphi$ holds in every $E'\models T$ into which $E_0$ embeds over $F$.
If every such subextension is again a model of $T$, this is equivalent to $E_0\models\varphi$.
Without loss of generality, one can take $E_0/F$ finitely generated, see \cite[Lemma 4.3]{DDF}.
We observe (and use without further comment) that
the class of all field extensions (respectively, all separable field extensions) of $F$ as $\Lring(F)$-structures is elementary.

For future use, we state two lemmas regarding the essential fiber dimension in greater generality than what we will use here.
Lemma~\ref{lem:efd} for $\exists_n\exists_1$-sentences follows in the special case of $\exists_{n+1}$-sentences
directly from
\cite[Propositions 4.9, 4.11]{DDF}.
Lemma~\ref{lem:efd_sep} contains no ideas that are not in \cite{DDF}, but is not spelled out there.

\begin{lemma}\label{lem:efd}
Let $F$ be a field,
	$\Sigma$ the $\mathfrak{L}_{\rm ring}(F)$-theory of field extensions of $F$,
and $n\geq 1$.
Then every $\exists_n\exists_1$-$\Lring(F)$-sentence $\varphi$ has
${\rm efd}_\Sigma(\varphi)\leq n$.
\end{lemma}

\begin{proof}
It suffices to show that if $\varphi$ holds in an extension $E$ of $F$,
then it holds in a subextension $E_0$ of $E/F$ with ${\rm trdeg}(E_0/F)\leq n$.
For this we can assume without loss of generality that $\varphi$ is of the form
$\exists x_1,\dots,x_n \psi$
with $\psi$ of the form $\bigwedge_{i=1}^r\exists y_i \psi_i$ with $\psi_i$ quantifier-free.
Let $a_1,\dots,a_n\in E$ such that
$E\models\psi(\underline{a})$,
i.e.~for every $i$ there exists $b_i\in E$
with $E\models\psi_i(\underline{a},b_i)$.

First suppose that $F(\underline{a})$ is infinite.
In this case, for each $i$, if $b_i$ is transcendental over $F(\underline{a})$, then $F(\underline{a})\prec_\exists F(\underline{a})(b_i)$ \cite[Example 3.1.2]{Ershov}, and therefore there exists $c_i\in F(\underline{a})$ with $F(\underline{a})\models\psi_i(\underline{a},c_i)$.
And if $b_i$ is algebraic over $F(\underline{a})$, let $c_i:=b_i$.
Then $E_0:=F(\underline{a},\underline{c})\models\varphi$
and
${\rm trdeg}(E_0/F)={\rm trdeg}(F(\underline{a})/F)\leq n$,
as claimed.

If on the other hand $F(\underline{a})$ is finite, 
then either ${\rm trdeg}(F(\underline{a},\underline{b})/F)=0$, in which case we can take $E_0=F(\underline{a},\underline{b})$,
or there exists at least one $b_i$ which is transcendental over $F(\underline{a})$, say $b_1$,
and then we argue as above that $\varphi$ holds in an extension $E_0:=F(\underline{a},b_1,c_2,\dots,c_r)$ with $c_2,\dots,c_r$ algebraic over $F(\underline{a},b_1)$,
so that ${\rm trdeg}(E_0/F)\leq 1\leq n$.
\end{proof}

\begin{lemma}\label{lem:efd_sep}
Let $F$ be a field and
	$T$ some $\Lring(F)$-theory of separable field extensions of $F$.
Then every $\Lring(F)$-sentence $\varphi$ with ${\rm efd}_T(\varphi)\leq n$ is equivalent modulo $T$ to an $\exists_{n+1}$-$\Lring(F)$-sentence.
\end{lemma}

\begin{proof}
Suppose that $\varphi$ holds in a separable extension $E$ of $F$.
The assumption ${\rm efd}_T(\varphi)\leq n$ implies that there is a finitely generated subextension $E_0$ of $E/F$ with ${\rm trdeg}(E_0/F)\leq n$ such that $\varphi$ holds in every $E'\models T$ into which $E_0$ embeds.
As $E/F$ is separable, so is $E_0/F$, so since it is also finitely generated, 
it is generated by $n+1$ elements (Lemma \ref{lem:prim_elem}).
So since $\varphi$ holds in every $E'\models T$ into which $E_0$ embeds,
this shows that $\varphi$ is equivalent modulo $T$ to an
$\exists_{n+1}$-sentence,
cf.~\cite[Proposition 2.8, Remark 3.1]{DDF}.
\end{proof}

\begin{proposition}\label{prop:efd}
Let $K,L$ be two separable extensions of a field $C$.
Then ${\rm Th}_{\exists_n\exists_1}(K,C)\subseteq{\rm Th}_{\exists_n\exists_1}(L,C)$ if and only if 
${\rm Th}_{\exists_{n+1}}(K,C)\subseteq{\rm Th}_{\exists_{n+1}}(L,C)$.
\end{proposition}

\begin{proof}
The implication from left to right is trivial.
Conversely,
assume that ${\rm Th}_{\exists_{n+1}}(K,C)\subseteq{\rm Th}_{\exists_{n+1}}(L,C)$
and let $\varphi\in{\rm Th}_{\exists_n\exists_1}(K,C)$.
Without loss of generality, $\varphi$ is an $\exists_n\exists_1$-sentence.
By Lemma~\ref{lem:efd},
${\rm efd}_\Sigma(\varphi)\leq n$,
where
	$\Sigma$ is the $\Lring(C)$-theory of all field extensions of $C$.
In particular, if
	$T$ denotes the $\Lring(C)$-theory of separable extensions of $C$,
${\rm efd}_T(\varphi)\leq n$
(cf.~\cite[Lemma 4.4]{DDF}),
so Lemma~\ref{lem:efd_sep} implies that $\varphi$ is equivalent modulo $T$ to an $\exists_{n+1}$-$\Lring(C)$-sentence $\psi$.
So since $K/C$ is separable, $K\models\psi$, therefore $L\models\psi$, and thus, since also $L/C$ is separable, $L\models\varphi$.
\end{proof}

\begin{remark}
Since all extensions $E/F$ occurring in the formulation of \Rfour\ are necessarily separable,
this immediately implies that if \Rfour\ holds for the $\exists_{n+1}$-fragment, i.e.~\RfourNplusone\ holds,
then \Rfour\ also holds for the $\exists_n\exists_1$-fragment.
However, for clarity of presentation, we state all our results under the same assumption \RfourN\ respectively \RfourNplusone.
\end{remark}

\begin{theorem}\label{thm:ENE1}
Let $(C,u)$ be an equicharacteristic $\ZZ$-valued field with uniformizer $\pi$ such that $\mathcal{O}_u$ is excellent. 
Let $(K,v)$ and $(L,w)$
be henselian extensions of $(C,u)$ with uniformizer $\pi$ such that $Kv/Cu$ and $Lw/Cu$ are separable.
 If ${\rm Th}_{\exists_n\exists_1}(Kv,Cu)\subseteq{\rm Th}_{\exists_n\exists_1}(Lw,Cu)$
 and \RfourNplusone\ holds,
then ${\rm Th}_{\exists_n\exists_1}(K,C)\subseteq{\rm Th}_{\exists_n\exists_1}(L,C)$.
\end{theorem}

\begin{proof}
The assumption ${\rm Th}_{\exists_n\exists_1}(Kv,Cu)\subseteq{\rm Th}_{\exists_n\exists_1}(Lw,Cu)$  in particular implies (in fact, is equivalent by Proposition~\ref{prop:efd}) that
${\rm Th}_{\exists_{n+1}}(Kv,Cu)\subseteq{\rm Th}_{\exists_{n+1}}(Lw,Cu)$, 
as $Kv/Cu$ and $Lw/Cu$ are separable.
Therefore, Theorem~\ref{thm:En}(1a) (assuming \RfourNplusone) shows that
${\rm Th}_{\exists_{n+1}}(K,C)\subseteq{\rm Th}_{\exists_{n+1}}(L,C)$.
But the assumptions that $\mathcal{O}_u$ is excellent and $Kv/Cu$ respectively $Lw/Cu$ are separable imply that also $K/C$ and $L/C$ are separable (see \cite[proof of Lemma 3.6/2]{BLR} and also the footnote on p.~2027 in \cite{ADF23}),
so Proposition \ref{prop:efd} implies that
${\rm Th}_{\exists_n\exists_1}(K,C)\subseteq{\rm Th}_{\exists_n\exists_1}(L,C)$.
\end{proof}

\begin{remark}
We point out that we do not prove a statement analogous to Theorem~\ref{thm:En}(1b),
i.e.~that if ${\rm Th}_{\exists_{n-1}\exists_1}(Kv,Cu)\subseteq{\rm Th}_{\exists_{n-1}\exists_1}(Lw,Cu)$
 and \RfourNplusone\ holds,
then ${\rm Th}_{\exists_{n-1}\exists_1}(K,v,C)\subseteq{\rm Th}_{\exists_{n-1}\exists_1}(L,w,C)$.
Namely, from Proposition~\ref{prop:val_to_ring} one only gets that every
$\exists_{n-1}\exists_1$-$\mathfrak{L}_{\rm val}(C)$-sentence is equivalent modulo $\Hpi$
to a $\exists_{n-1}\exists_2$-$\mathfrak{L}_{\rm ring}(C)$-sentence (in the notation of \cite{AFAE}),
but it is unclear whether it is also equivalent to a
$\exists_{n}\exists_1$-$\mathfrak{L}_{\rm val}(C)$-sentence.
\end{remark}

\section{Monotonicity for the $\exists^n$-fragment}\label{sec:Euppern}

\noindent
Finally, we consider the fragment which in the notation introduced in \cite[Definition 2.3]{AFAE} is denoted $\exists^n$: 
For a language $\mathfrak{L}$, an $\exists^n$-$\mathfrak{L}$-formula 
is an $\mathfrak{L}$-formula of the form $\psi$ or $\exists x\psi$ with
$\psi$ in the $\exists^{n-1}$-fragment, i.e.~a positive boolean combination of $\exists^{n-1}$-$\mathfrak{L}$-formulas, defined inductively, with $\exists^0$ being the quantifier-free formulas. 
Note that the $\exists^1$-formulas coincide with the $\exists_1$-formulas, 
the $\exists^2$-formulas coincide with the $\exists_1\exists_1$-formulas from the previous section, and for $n>1$ the set of $\exists_n\exists_1$-formulas is properly contained in the set of $\exists^{n+1}$-formulas.
Again, we denote by ${\rm Th}_{\exists^n}(K)$ the set of positive boolean combinations of $\exists^n$-$\mathfrak{L}$-formulas that hold in $K$.

\begin{proposition}\label{prop:E^n}
Let $\mathfrak{L}$ be a language and $M,N$ $\mathfrak{L}$-structures.
Then ${\rm Th}_{\exists^{n+1}}(M)\subseteq{\rm Th}_{\exists^{n+1}}(N)$ if and only if 
for every $a\in M$ there exists $N\prec N^*$ and $b\in N$ such that
${\rm Th}_{\exists^n}(M,a)\subseteq{\rm Th}_{\exists^n}(N^*,b)$,
where $(M,a)$, $(N,b)$ are $\mathfrak{L}(c)$-structures with $c$ a new constant symbol.
\end{proposition}

\begin{proof}
First suppose that ${\rm Th}_{\exists^{n+1}}(M)\subseteq{\rm Th}_{\exists^{n+1}}(N)$
and let $a\in M$.
Every $\varphi\in{\rm Th}_{\exists^n}(M,a)$ is of the form $\psi(a)$ with $\psi(x)$ a positive boolean combination of $\exists^n$-$\mathfrak{L}$-formulas, and $\exists x\psi(x)$ is an $\exists^{n+1}$-$\mathfrak{L}$-sentence that holds in $M$, so by assumption also in $N$.
Thus, as ${\rm Th}_{\exists^n}(M,a)$ is by definition closed under conjunctions,
by the compactness theorem there exists $N\prec N^*$ and $b\in N^*$ for which all of these $\psi$ hold, hence ${\rm Th}_{\exists^n}(M,a)\subseteq{\rm Th}_{\exists^n}(N^*,b)$.

For the other implication, let $\varphi\in{\rm Th}_{\exists^{n+1}}(M)$.
Without loss of generality, $\varphi$ is of the form $\exists x \psi(x)$ with $\psi$ a positive boolean combination of $\exists^n$-formulas, and there exists $a\in M$ with $M\models\psi(a)$.
The assumption is that there exists $N\prec N^*$ and $b\in N$ such that ${\rm Th}_{\exists^n}(M,a)\subseteq{\rm Th}_{\exists^n}(N^*,b)$, and therefore $N^*\models\psi(b)$,
in particular $N^*\models\exists x\psi(x)$ and therefore also $N\models\varphi$.
\end{proof}

\begin{corollary}\label{cor:E^n}
Let $(K,v),(L,w)$ be extensions of a valued field $(C,u)$. 
Then
${\rm Th}_{\exists^{n+1}}(K,v,C)\subseteq{\rm Th}_{\exists^{n+1}}(L,w,C)$ if and only if
for every $a\in K$ there exists an elementary extension $(L,w)\prec (L^*,w^*)$ and 
an embedding $\iota\colon (C(a),v)\rightarrow(L^*,w^*)$ over $C$
such that
$$
 {\rm Th}_{\exists^{n}}(K,v,C(a))\subseteq{\rm Th}_{\exists^n}(L^*,w^*,\iota(C(a))).
$$ 
\end{corollary}

\begin{proof}
This immediately follows from Proposition \ref{prop:E^n} for the language $\mathfrak{L}=\mathfrak{L}_{\rm val}(C)$, as already ${\rm Th}_{\exists^0}(K,v,C,a)$ determines the isomorphism type of $C(a)$.    
\end{proof}

\begin{lemma}\label{lem:claim}
Let $C$ be a field, 
let $K,L$ be extensions of $C$,
let $C'=\overline{C}^K$ be the relative algebraic closure of $C$ in $K$,
and let $n\geq 1$.
If ${\rm Th}_{\exists^n}(K,C)\subseteq{\rm Th}_{\exists^n}(L,C)$,
then there exists an embedding $\alpha\colon C'\rightarrow_C L$ such that
we have ${\rm Th}_{\exists^n}(K,C')\subseteq{\rm Th}_{\exists^n}(L,\alpha(C'))$.
\end{lemma}

\begin{proof}
The proof is by induction on $n$.
The case $n=1$ is immediate from Lemma \ref{lem:E1}, as $\overline{C'}^K=C'$.
    Now suppose the statement holds for $n\geq 1$.
    We want to use Corollary \ref{cor:E^n} (with all valuations trivial) to show that it holds for $n+1$, 
    so assume that ${\rm Th}_{\exists^{n+1}}(K,C)\subseteq{\rm Th}_{\exists^{n+1}}(L,C)$
    and let $a\in K$.
    By Corollary \ref{cor:E^n}, 
    there exists $L\prec L^*$ and an embedding
    $\iota\colon C(a)\rightarrow L^*$ such that
    ${\rm Th}_{\exists^{n}}(K,C(a))\subseteq{\rm Th}_{\exists^{n}}(L^*,\iota(C(a)))$.
    We identify $C(a)$ with $\iota(C(a))$ to assume without loss of generality that $\iota={\rm id}_{C(a)}$.
	The inductive hypothesis then gives that
	there exists an embedding $\alpha\colon C(a)'\rightarrow_{C(a)} L^{*}$,
    where $C(a)'$ is the relative algebraic closure of $C(a)$ in $K$,
    such that
    ${\rm Th}_{\exists^{n}}(K,C(a)')\subseteq{\rm Th}_{\exists^{n}}(L^*,\alpha(C(a)'))$.
	As $C'(a)\subseteq C(a)'$ and $a\in K$ was arbitrary, this shows that
	${\rm Th}_{\exists^{n+1}}(K,C')\subseteq{\rm Th}_{\exists^{n+1}}(L,\alpha(C'))$,
    by Corollary \ref{cor:E^n}.
\end{proof}

\begin{theorem}
    Let $(C,u)$ be a $\ZZ$-valued field of residue characteristic zero with uniformizer $\pi$.
Let $(K,v)$ and $(L,w)$
be henselian extensions of $(C,u)$ with uniformizer $\pi$,
and let $n\geq 0$.
If ${\rm Th}_{\exists^n}(Kv,Cu)\subseteq{\rm Th}_{\exists^n}(Lw,Cu)$,
then ${\rm Th}_{\exists^n}(K,v,C)\subseteq{\rm Th}_{\exists^n}(L,w,C)$.
\end{theorem}

\begin{proof}
Let $k=Kv$, $l=Lw$, $\kappa=Cu$ and 
note that the completion $(\hat{C},\hat{u})$ of $(C,u)$ is isomorphic to $(\ps{\kappa},v_t)$, 
via an isomorphism $\alpha\colon(\hat{C},\hat{u})\rightarrow(\ps{\kappa},v_t)$ with $\alpha(\pi)=t$ and which induces the identity on $\kappa$.
As $(\ps{k}v_t,\alpha(C)v_t)=(Kv,Cu)$,
we get
${\rm Th}_\exists(\ps{k},v_t,\alpha(C))={\rm Th}_\exists(K,v,C)$ 
(see e.g.~\cite[Proposition 4.11 and Remark 4.18]{ADF23} and Remark \ref{rem:char0}), 
so that we can assume without loss of generality that
$(K,v)=(\ps{k},v_t)$,
and similarly that
$(L,w)=(\ps{l},v_t)$, and $\pi=t$.
Moreover, then we can also assume that $(C,u)=(\hat{C},\hat{u})=(\ps{\kappa},v_t)$.

The  proof is now an induction on $n$.
The base case $n=0$ is trivial.
Next suppose that the statement holds for $n\geq0$,
and assume that 
\begin{equation}\label{eqn:n+1}
{\rm Th}_{\exists^{n+1}}(k,\kappa)\subseteq{\rm Th}_{\exists^{n+1}}(l,\kappa).
\end{equation}
Let $C'=\overline{C}^K$, $C''=\overline{C}^L$ be the relative algebraic closure of $C$ in $K$ respectively $L$.
Since
both
$k/\kappa$ and $l/\kappa$ are separable
	(as the characteristic is zero)
and both $(K,v)$ and $(L,w)$ are henselian,
$\kappa':=C'v$ is relatively algebraically closed in $k$, and $C''w$ is relatively algebraically closed in $l$.
By Lemma~\ref{lem:claim}, we can identify $\kappa'$ with a subextension of $l/\kappa$
such that \eqref{eqn:n+1} holds with $\kappa$ replaced by $\kappa'$.
Note that we then have $C'v\subseteq C''w$.
Extend $u$ to a valuation $\bar{u}$ on $\overline{C}$, and embed $(\overline{C},\bar{u})$ into $\overline{K}$ and $\overline{L}$ together with a choice of extension of $v$ respectively $w$.
Then the inertia field of $\bar{u}/u$ contains both $C'$ and $C''$,
as $K,L$ (and therefore $C',C''$) are unramified and separable extensions of the discretely valued field $C$,
cf.~\cite[Theorem 5.2.9(1) and Theorem 3.3.5]{EP}.
As there is an inclusion preserving bijection between the subextensions of the inertia field and the separable algebraic extensions of the residue field \cite[Theorem 5.2.7(2)]{EP}, and $C'v\subseteq C''w$,
we conclude that $C'\subseteq C''$, and so we can view $(C',v)$ as a valued subfield of $(L,w)$.
Since we already established that \eqref{eqn:n+1} holds for $\kappa'$ instead of $\kappa$,
we then can assume without loss of generality that $\kappa=\kappa'$ and $C=C'$,
and also that $(C,u)$ is again complete, since $(K,v)$ and $(L,w)$ are complete.
We can now apply the inductive hypothesis to \eqref{eqn:n+1} to conclude that
${\rm Th}_{\exists^n}(K,v,C)\subseteq{\rm Th}_{\exists^n}(L,w,C)$.

We want to use Corollary \ref{cor:E^n} to show that ${\rm Th}_{\exists^{n+1}}(K,v,C)\subseteq{\rm Th}_{\exists^{n+1}}(L,w,C)$, so let $a\in K$.
If $a\in C$, we can take $\iota={\rm id}_{C}$ and get
${\rm Th}_{\exists^n}(K,v,C(a))={\rm Th}_{\exists^n}(K,v,C)\subseteq{\rm Th}_{\exists^n}(L,w,C)=
{\rm Th}_{\exists^n}(L,w,\iota(C(a)))$, as desired.
So suppose that $a\notin C$.
Since $(K,v)/(C,v)$ is unramified and $(C,v)$ is complete and therefore maximal, the residue extension $C(a)v/Cv$ is nontrivial and,
as $Cv=\kappa$ is algebraically closed in $Kv=k$, therefore transcendental.
By the ruled residue theorem \cite{Ohm}, $C(a)v$ is a simple transcendental extension of a finite extension of $Cv$,
so again since $\kappa$ is algebraically closed in $k$, a simple transcendental extension of $\kappa$ itself, say 
$C(a)v=\kappa(\bar{b})$ for some $b\in C(a)$. 
	We observe that $(C(a),v)/(C(b),v)$ is immediate, and therefore so is $(C(a)^{h},v)/(C(b),v)$. 
	In particular $(C(b)^{h}(a),v)/(C(b)^{h},v)$ is immediate, which implies that $a\in C(b)^{h}$ since in residue characteristic zero henselian implies algebraically maximal
	\cite[Theorem 4.1.10]{EP}.
By Corollary \ref{cor:E^n} (applied with all valuations trivial), 
\eqref{eqn:n+1}
implies that there exists $l\prec l^*$ and an embedding $\bar{\iota}\colon \kappa(\bar{b})\rightarrow l^*$ over $\kappa$ such that
\begin{equation}\label{eqn:n}
 {\rm Th}_{\exists^{n}}(k,\kappa(\bar{b}))\subseteq{\rm Th}_{\exists^{n}}(l^*,\bar{\iota}(\kappa(\bar{b}))).
\end{equation}
Let $(L,w)\prec(L^*,w^*)$ be any elementary extension with $l^*\subseteq L^*w^*$
(for example, we can take $(L^*,w^*)=(l^*(\!(t)\!),v_t)$ by the classical Ax-Kochen--Ershov theorem, see e.g.~\cite[Theorem~4.3.2]{Ershov}).
Let $c\in L^*$ such that $\bar{c}=\bar{\iota}(\bar{b})$ and note that $\bar{c}$ is transcendental over $Cv$,
hence $c$ is transcendental over $C$,
and so $(C(c),w^*)$ is the Gauss extension of $(C,w)$
and therefore there exists an isomorphism $(C(b),v)\rightarrow(C(c),w^*)$ over $C$
\cite[Corollary 2.2.2]{EP},
which uniquely extends to an embedding
$\iota\colon(C(b)^h,v)\rightarrow(L^*,w^*)$.
Note that $C(a)v=\kappa(\bar{b})$ and $C(c)w^*=\kappa(\bar\iota(\bar{b}))$, and so 
\eqref{eqn:n}
together with the inductive hypothesis implies that
${\rm Th}_{\exists^n}(K,v,C(a))\subseteq{\rm Th}_{\exists^n}(L^*,w^*,\iota(C(a)))$,
which concludes the proof by Corollary~\ref{cor:E^n}.
\end{proof}

\section{Applications}

\noindent
We say that {\em resolution of singularities holds up to dimension $n$} if for every field $K$ and every $K$-variety $X$
(by which we mean an integral separated $K$-scheme of finite type)
with ${\rm dim}(X)\leq n$ there exists a regular $K$-variety $Y$ and a proper birational morphism $Y\rightarrow X$.
For fields of characteristic zero, this was famously proven for all $n$ by Hironaka.
For general fields, the currently best result is the following:

\begin{theorem}[Cossart--Piltant]
Resolution of singularities holds up to dimension $3$.
\end{theorem}

\begin{proof}
This is a special case of \cite[Theorem 1.1]{CP},
as every $K$-variety is excellent \cite[Tag 07QW]{Stacks}.
\end{proof}

\begin{lemma}\label{prop}
Suppose that resolution of singularities holds up to dimension $n$.
Let $K$ be a large field, $F/K$ an extension, $v$ a valuation on $F$, trivial on $K$ and with $Fv=K$.
Let $T$ be an $\Lring(K)$-theory 
of field extensions of $K$
of which every intermediate field of $F/K$ that is relatively algebraically closed in $F$ is a model.
If $\varphi$ is an existential $\mathfrak{L}_{\rm ring}(K)$-sentence with ${\rm efd}_{T}(\varphi)\leq n$,
and $F\models\varphi$, then $K\models\varphi$.
\end{lemma}

\begin{proof}
By the definition of essential fiber dimension
and \cite[Lemma 4.3]{DDF},
there exists a finitely generated subextension $K'/K$ of $F/K$ of transcendence degree at most $n$,
such that $M\models\varphi$ for every model $M$ of $T$ into which $K'$ embeds.
In particular, $\varphi$ holds in the relative algebraic closure $K''$ of $K'$ in $F$,
and therefore already in a finite subextension $E$ of $F/K'$.

As $E/K$ is finitely generated,
there exists a proper (e.g.~projective) $K$-variety $X$ with $K(X)=E$.
By the assumption of resolution of singularities up to dimension $n$,
we can assume without loss of generality that $X$ is regular.
Since $X$ is proper, $w:=v|_E$ is centered at a point $P$ of $X$, so the regular local ring $\mathcal{O}_{X,P}$ is dominated by the valuation ring $\mathcal{O}_w$ and therefore has residue field $K$.
Thus $P\in X(K)$, and therefore $P\in X_{\rm smooth}(K)$ \cite[Proposition 17.15.1]{EGAIV4}.
Since $K$ is large, this implies that $X(K)$ is Zariski-dense in $X$,
which in turn is equivalent to $K\prec_\exists E$ (cf.~\cite[Fact 2.3]{Popsurvey}).
Thus, $E\models\varphi$ implies $K\models\varphi$.
\end{proof}

\begin{proposition}\label{prp:R4n+1}
If resolution of singularities holds up to dimension $n$,
then \RfourNplusone\ holds.
\end{proposition}

\begin{proof}
Let $E/F$ be as in the statement of \Rfour,
and let $\varphi$ be an $\exists_{n+1}$-$\Lring(F)$-sentence,
so in particular $\varphi$ is an $\exists_{n}\exists_{1}$-$\Lring(F)$-sentence.
	Let $T$ be the $\Lring(F)$-theory of field extensions of $F$.
By Lemma \ref{lem:efd},
${\rm efd}_{T}(\varphi)\leq n$
so the claim follows from Lemma~\ref{prop}.
\end{proof}

\begin{corollary}\label{cor:E4}
\Rfourfour\ holds.
\end{corollary}

\begin{corollary}\label{cor:applications}
The following theories are decidable:
\begin{enumerate}
	\item
		${\rm Th}_{\exists_3}(\ps{\FF_{q}},v_t,t)$
		and
		${\rm Th}_{\exists_3}(\ps{\overline{\FF_{q}}},v_t,t)$
	\item
		${\rm Th}_{\exists_4}(\ps{\FF_{q}},t)$
		and
		${\rm Th}_{\exists_4}(\ps{\overline{\FF_{q}}},t)$
	\item
		${\rm Th}_{\exists_3\exists_1}(\ps{\FF_{q}},t)$
		and
		${\rm Th}_{\exists_3\exists_1}(\ps{\overline{\FF_{q}}},t)$
\end{enumerate}
\end{corollary}

\begin{proof}
This follows by a standard argument 
from Theorems \ref{thm:En}(1) and \ref{thm:ENE1} 
by enumerating proofs
from the recursive axioms of $\Hepi$
together with the quantifier-free diagram of 
the computable valued field $(C,u)=(\mathbb{F}_p(t),v_t)$,
and the decidable (existential/universal) $\mathfrak{L}_{\rm ring}(\mathbb{F}_p)$-theory of the residue field $\mathbb{F}_q$ respectively $\overline{\mathbb{F}_q}$,
see for example the proof of \cite[Theorem 4.12]{ADF23}.
	For an account of this standard argument in an abstract setting, 
	we refer the reader to~\cite[\S2]{AFfragments}.
\end{proof}

\begin{remark}\label{rem:DS}
It seems plausible that Corollary \ref{cor:applications}(2) can be deduced (again using \cite{CP}) without much effort from \cite{DS}, 
as the algorithm given there (see also \cite[Section 3]{ADF23} for a discussion and some corrections) needs resolution of singularities only for subvarieties of the variety described by the input sentence.
However, it is less clear how to obtain  Corollary~\ref{cor:applications}(1) this way,
and our approach 
gives not only (relative) decidability but an axiomatization and hence understanding of the theory, from which one can moreover deduce for example further model theoretic consequences (as in \cite{ADF23}) as well as results uniform across residue theories (as in \cite{AFfragments}).
\end{remark}

\begin{corollary}
Let $(K,v)$ be a henselian valued field with uniformizer $p$ and residue field $k=\mathbb{F}_p(\!(t)\!)$, e.g.~the quotient field of the Cohen--Witt ring of $k$, cf.~\cite[\S7]{AJ22}, and let $s\in K$ be an element with residue $t$.
Then $\Th_{\exists_{4}}(K,v,s)$ is decidable.
\end{corollary}

\begin{proof}
 Let $(C,u)=(\mathbb{Q}(s),v)$ and note that since $\bar{s}=t$ is transcendental, $u$ is the Gauss extension of $v_p$ \cite[Corollary 2.2.2]{EP}, in particular $Cu=\mathbb{F}_p(t)$.
	The claim now follows from Theorem~\ref{thm:En}(\ref{case3}), again 
	by enumerating proofs from axioms of the  
	$\Lval$-theory of unramified henselian valued fields $(L,w)$ of mixed characteristic together with the quantifier-free diagram of $(C,u)$
    (which is decidable since $u$ is the Gauss extension of $v_p$),
    and axioms that state that ${\rm Th}_{\exists_4}(Lw,Cu)={\rm Th}_{\exists_4}(\ps{\FF_{p}},\mathbb{F}_p(t))$
    (which is decidable by Corollary~\ref{cor:applications}(2)).
\end{proof}

\section*{Acknowledgements}

\noindent
The authors would like to extend their sincere thanks to
Franz-Viktor Kuhlmann for his suggestion and encouragement to use the desingularization of threefolds from \cite{CP} to obtain unconditional results about existential fragments.
Thanks are also due to Philip Dittmann, Franz-Viktor Kuhlmann, and to the referee for many helpful comments on earlier versions.

Most of this research was done while the authors were participating in the trimester on ``Definability, decidability, and computability'' at the Hausdorff Institute Bonn, funded by the Deutsche Forschungsgemeinschaft (DFG, German Research Foundation) under Germany‘s Excellence Strategy – EXC-2047/1 – 390685813.
S.~A.\ and A.~F.\ were supported by the Institut Henri Poincaré.
S.~A.\ was also 
supported by
the ANR-DFG project ``AKE-PACT'' (ANR-24-CE92-0082)
and
by ``Investissement d'Avenir'' launched by the French Government and implemented by ANR (ANR-18-IdEx-0001) as part of its program ``Emergence''.


\begin{thebibliography}{DDF21}
\bibitem[ADF23]{ADF23}
S.~Anscombe, P.~Dittmann and A.~Fehm.
\newblock \href{https://doi.org/10.2140/ant.2023.17.2013}{Axiomatizing the existential theory of $\ps{\FF_{q}}$}.
\newblock {\em Algebra \& Number Theory}  17(11):2013--2032, 2023.

\bibitem[ADJ24]{ADJ24}
S.~Anscombe, P.~Dittmann and F.~Jahnke.
\newblock \href{https://doi.org/10.1090/tran/9271}{Ax–Kochen–Ershov principles for finitely ramified henselian fields}.
\newblock {\em Trans.~Amer.~Math.~Soc.} 377(12):8963--8988, 2024.

\bibitem[AF16]{AF16}
S.~Anscombe and A.~Fehm.
\newblock \href{https://msp.org/ant/2016/10-3/ant-v10-n3-p06-s.pdf#06TITLEPAGE}{The existential theory of equicharacteristic henselian valued fields}.
\newblock {\em Algebra \& Number Theory} 10(3), 2016.

\bibitem[AF25]{AFAE}
S.~Anscombe and A.~Fehm.
\newblock \href{https://doi.org/10.1017/S0013091525100771}{Universal-existential theories of fields}.
\newblock {{\em Proceedings of the Edinburgh Mathematical Society}, published online 2025:1--30.}

\bibitem[AF26]{AFfragments}
S.~Anscombe and A.~Fehm.
\newblock Interpretations of syntactic fragments of theories of fields.
\newblock To appear in {\em Israel J.\ Math}, 2026. \href{https://doi.org/10.48550/arXiv.2312.17616}{arXiv:2312.17616 [math.LO]}

\bibitem[AJ22]{AJ22}
S.~Anscombe and F.~Jahnke.
\newblock \href{https://doi.org/10.5802/cml.84}{The model theory of Cohen rings}.
\newblock {\em Conf.~Math.} 14(2):1--28, 2022.

\bibitem[BF13]{BSFsurvey}
L.~Bary-Soroker and A.~Fehm.
\newblock Open problems in the theory of ample fields.
\newblock In {\em Geometric and differential {G}alois theories}, volume~27 of {\em S\'emin. Congr.}, pages 1--11. Soc. Math. France, Paris, 2013.

\bibitem[BLR90]{BLR}
S.~Bosch, W.~Lütkebohmert and M.~Raynaud.
\newblock \href{https://doi.org/10.1007/978-3-642-51438-8}{\em Néron models}.
\newblock Springer, 1990.

\bibitem[CK90]{CK}
C.~C.~Chang and H.~J.~Keisler.
\newblock {\em Model Theory.}
\newblock Third edition. Studies in Logic and the Foundations of Mathematics, 73. North-Holland Publishing Co., Amsterdam, 1990.

\bibitem[CP19]{CP}
V.~Cossart and O.~Piltant.
\newblock \href{https://doi.org/10.1016/j.jalgebra.2019.02.017}{Resolution of singularities of arithmetical threefolds}.
\newblock {\em J.~Algebra} 529:268--535, 2019.

\bibitem[DDF21]{DDF}
N.~Daans, P.~Dittmann and A.~Fehm.
\newblock Existential rank and essential dimension of diophantine sets.
\newblock \href{https://arxiv.org/abs/2102.06941}{arXiv:2102.06941 [math.NT]}, 2021.

\bibitem[DS03]{DS}
J.~Denef and H.~Schoutens.
\newblock {On the decidability of the existential theory of $\ips{\FF_{p}}$}.
\newblock In {\em Valuation theory and its applications}, (Saskatoon, 1999), vol.~II, edited by
F.-V. Kuhlmann et al., Fields Inst. Commun. 33, Amer. Math. Soc., Providence, RI,
pp.~43--60,
2003.

\bibitem[Dri14]{Dries}
L.~van den Dries.
\newblock \href{https://doi.org/10.1007/978-3-642-54936-6_4}{Lectures on the model theory of valued fields}.
\newblock In: {\em Model theory in algebra, analysis and arithmetic}, pp.~55--157. Springer, 2014.

\bibitem[EP05]{EP}
A.~Engler and A.~Prestel.
\newblock \href{https://doi.org/10.1007/3-540-30035-X}{\em Valued Fields}.
\newblock Springer Monographs in Mathematics.
\newblock Springer, 2005.

\bibitem[Ers01]{Ershov}
Yu.~L.~Ershov.
\newblock \href{https://doi.org/10.1007/978-1-4615-1307-0}{\em Multi-Valued Fields}.
\newblock Springer, 2001.

\bibitem[Feh11]{Feh11}
A.~Fehm.
\newblock \href{https://doi.org/10.1007/s00229-010-0415-8}{Embeddings of function fields into ample fields}.
\newblock {\em Manuscripta Math.} 134:533–544, 2011.

\bibitem[FJ08]{FJ}
M.~D.~Fried and M.~Jarden.
\href{https://doi.org/10.1007/978-3-540-77270-5}{\em Field Arithmetic}.
\newblock Third Edition. Springer, 2008.

\bibitem[Gro67]{EGAIV4}
A.~Grothendieck.
\newblock {\'El\'ements de g\'eom\'etrie alg\'ebrique (r\'edig\'e avec la
  cooperation de Jean Dieudonn\'e): IV. \'Etude locale des sch\'emas et des
  morphismes des sch\'emas, Quatri\`eme partie}.
\newblock {\em {Publ. Math. IHES}} ({32}):{5--361}, 1967.

\bibitem[Kuh04]{Kuhlmann}
F.-V.~Kuhlmann.
\newblock \href{https://doi.org/10.1016/j.aim.2003.07.021}{Places of algebraic function fields in arbitrary characteristic}.
\newblock {\em Adv.\ Math.} 188:399--424, 2004.

\bibitem[Ohm83]{Ohm}
J.~Ohm.
\newblock \href{https://doi.org/10.2307/2045053}{The ruled residue theorem for simple transcendental extensions of valued fields}.
\newblock {\em Proc.\ Amer.\ Math.\ Soc.} 89(1), 1983.

\bibitem[Pop96]{Pop}
F.~Pop.
\newblock \href{https://doi.org/10.2307/2118581}{Embedding problems over large fields.}
\newblock {\em Ann.~of Math.}, 144:1--34, 1996.

\bibitem[Pop14]{Popsurvey}
F.~Pop.
\newblock Little survey on large fields—old \& new.
\newblock In {\em Valuation theory in interaction. Proceedings of the Second International Conference and Workshop on Valuation Theory held in Segovia and El Escorial, July 18–29, 2011}, edited by A.~Campillo et.\ al., EMS Ser. Congr. Rep., pp.~432--463, 2014.

\bibitem[PD11]{PD}
A.~Prestel and C.~N.~Delzell.
\newblock \href{https://doi.org/10.1007/978-1-4471-2176-3}{\em Mathematical Logic and Model Theory}.
\newblock Springer, 2011.

\bibitem[Stacks]{Stacks}
The {Stacks Project Authors}.
\newblock {\textit{Stacks Project}}, \url{https://stacks.math.columbia.edu}, 2023.

\end{thebibliography}
\end{document}